\documentclass[a4paper,12pt]{amsart}%dvipdfmx,
\RequirePackage{plautopatch}
\RequirePackage[l2tabu, orthodox]{nag}
\usepackage{bxtexlogo}
\usepackage{amsfonts}
\usepackage{bbm}
\usepackage{amssymb, amsmath,stmaryrd,mathrsfs,mathtools}
\usepackage{tikz,float}
\usepackage{tikz-cd}
\usepackage{amsthm}
\usepackage{enumerate}
\usepackage{lmodern}

\usetikzlibrary{calc}
\usepackage{url}
\usepackage[hidelinks]{hyperref}
\usepackage{graphicx}
\usepackage {diagbox}
\usepackage{cleveref}
\newtheorem{thm}{Theorem}

\newtheorem{lem}[thm]{Lemma}

\newtheorem{ques}{Question}

\newcounter{enuAlph}

\newtheorem{teorema}[enuAlph]{Theorem}

\theoremstyle{definition}

\newtheorem{dfn}[thm]{Definition}
\newtheorem{rem}[thm]{Remark}

\newtheorem*{acknowledgements}{Acknowledgements}

\numberwithin{thm}{section} 
\numberwithin{equation}{section}

\usepackage{etoolbox}

\renewcommand{\a}{\alpha}

\newcommand{\p}{\mathbb{P}}

\newcommand{\qd}{\dot{\mathbb{Q}}}

\DeclareMathAlphabet{\mymathbb}{U}{BOONDOX-ds}{m}{n}
\newcommand{\zero}{\mymathbb{0}}

\newcommand{\oo}{{\omega}^\omega}
\newcommand{\ooo}{[{\omega}]^\omega}

\newcommand{\fsq}{\omega^{<\omega}}
\newcommand{\sq}{2^{<\omega}}

\newcommand{\on}{\mathpunct{\upharpoonright}}

\newcommand{\bb}{\mathfrak{b}}

\newcommand{\dd}{\mathfrak{d}}
\newcommand{\ee}{\mathfrak{e}}

\newcommand{\mm}{\mathfrak{m}}

\newcommand{\spl}{\mathfrak{s}}

\DeclareMathOperator{\non}{non}

\DeclareMathOperator{\add}{add}

\newcommand{\nonm}{\non(\mathcal{M})}

\newcommand{\addn}{\add(\mathcal{N})}

\newcommand{\R}{\mathbf{R}}

\DeclareMathOperator{\dom}{dom}
\DeclareMathOperator{\ran}{ran}

\DeclareMathOperator{\cf}{cf}

\newcommand{\prp}{\mathbb{PR}}

\newcommand{\str}{\mathrm{Str}}
\newcommand{\rp}{\sqsubset^\mathrm{r}}
\newcommand{\rsp}{\sqsubset^\mathrm{s}}
\newcommand{\rs}{\R^*}
\newcommand{\rss}{\R^{**}}

\newcommand{\tri}{\triangleleft_{*}}
\newcommand{\trii}{\triangleleft_{**}}

\newcommand{\spst}{\mathfrak{s}_{\mathrm{game}^*}^\mathrm{I}}
\newcommand{\spstt}{\mathfrak{s}_{\mathrm{game}^{**}}^\mathrm{I}}

\newcommand{\zp}{\sqsubset^{\mathrm{p}}_\zero}
\newcommand{\pdtd}{\sqsubset^{\mathrm{p}}}

\makeatletter
\renewcommand{\p@enumi}{}
\makeatother

\newcommand{\eeg}{\ee^\mathrm{G}}

\DeclareMathOperator{\Pred}{Pred}

\newcommand{\rel}{\sqsubset}

\newcommand{\scrA}{\mathcal{A}}
\newcommand{\Pow}{\mathcal{P}}
\newcommand{\seq}[1]{{\langle#1\rangle}}
\newcommand{\sIstar}{\mathfrak{s}_\mathrm{game^\ast}^{\mathrm{I}}}
\newcommand{\sIwstar}{\mathfrak{s}_{\mathrm{game^{\ast\ast}}}^\mathrm{I}}
\title{Evasion numbers via zero-prediction}
\author{Takashi Yamazoe}
\address{Graduate School of System Informatics, Kobe University,
	Rokko--dai 1--1, Nada--ku, 657--8501 Kobe, Japan}
\email{yamazoe@people.kobe-u.ac.jp}
\begin{document}
	\begin{abstract}
		%(Tentative) %In \cite{CGHY26}, 
		Cruz Chapital, Goto, Hayashi and the author showed that the game-theoretic variants $\spst$ and $\spstt$ of the splitting number $\spl$ are consistently different, although the corresponding two games differ only in a minor case.
		This result suggests that even if two relational systems $\R=\langle X,Y,\rel\rangle$, $\R^\prime=\langle X,Y,\rel^\prime\rangle$ are the same modulo a countable set $C\subseteq X$, 
		%their bounding numbers $\bb(\R)$ and $\bb(\R^\prime)$ can be separated.
		the associated cardinal invariants might be different.
		We study this phenomenon for the standard relational system of evasion and prediction and for a variation of it. We show that such a difference occurs for the standard one, but not for the variation.
	\end{abstract}
	\maketitle
	
	\section{Introduction}
	In \cite{CGHY26}, Cruz Chapital, Goto, Hayashi and the author introduced six game-theoretic variants of the splitting number $\spl$. We deal with two of them, $\spst$ and $\spstt$, defined as follows: 
	
	We define two kinds of games of length $\omega$ played by Player I and Player II.
	Fix a set $\scrA \subseteq \Pow(\omega)$.
	We call the game indicated by the following Table \ref{table:splittingstargame} the \textit{splitting* game} with respect to $\scrA$.
	
	\begin{table}[H]
		\caption{The splitting* game}
		\centering
		\begin{tabular}{l|llllll}\label{table:splittingstargame}
			Player I  & $i_0$ &       & $i_1$ &       & $\dots$ &     \\  \hline
			Player II &       & $j_0$ &       & $j_1$ &         & $\dots$
		\end{tabular}
	\end{table}
	
	Here, $\seq{i_k : k  \in \omega}$ and $\seq{j_k : k  \in \omega}$ are in $2^\omega$. %sequences of elements in $2$. %$\{0,1\}$.
	Player II wins when either Player I did not play $1$ infinitely often or
	\begin{align*}
		\{ k \in \omega : j_k = 1 \} \in \scrA \text{ and  splits } \{ k \in \omega : i_k = 1\}. 
	\end{align*}
	
	In the splitting* game, when both $\{ k \in \omega : j_k = 1 \}$ and $\{ k \in \omega : i_k = 1 \}$ are finite, 
	the winner is defined as Player II, which is not necessarily obvious. Hence, we just define another game by switching the winner in this case. Namely,
	the \textit{splitting** game} with respect to $\scrA$ follows the same rule as the splitting* game,
	but only the winning condition of Player II is replaced by:
	Player II wins when \underline{either} Player I did not play $1$ infinitely often \textit{and Player II did}, \underline{or}
	\begin{align*}
		\{ k \in \omega : j_k = 1 \} \in \scrA \text{ and  splits } \{ k \in \omega : i_k = 1\}. 
	\end{align*}
	Let:
	\begin{align*}
		\sIstar &= \min \{ |\scrA| : \text{Player I has no winning strategy} \\
		& \hspace{15ex} \text{for the splitting* game with respect to } \scrA \}, \text{ and} \\
		\sIwstar &= \min \{ |\scrA| : \text{Player I has no winning strategy} \\
		& \hspace{15ex} \text{for the splitting** game with respect to } \scrA \}.
	\end{align*}
	Thus, the splitting* game and the splitting** game are basically the same: the only difference appears when both players play finite sets, in which case the winner is Player II in the splitting* game but is Player I in the splitting** game. However, the cardinal invariants of the two games are different:
	
	\begin{thm}[{\cite[Theorem 3.17]{CGHY26}}]
		$\sIstar<\sIwstar$ is consistent.
	\end{thm}
	
	To analyze this separation result in more detail, %and to find a similar example, 
	recall the concept of relational systems:  
	\begin{dfn}
		\begin{itemize}
			\item $\R=\langle X,Y,\rel\rangle$ is a \textit{relational system} if $X$ and $Y$ are non-empty sets and $\rel \subseteq X\times Y$.
			%$\R$ is a \textit{relational system on the reals} if $X$ and $Y$ are sets of reals.
			\item We call an element of $X$ a \textit{challenge} and an element of $Y$ a \textit{response}, and say that $x$ \textit{is met by} $y$ when $x\rel y$.
			\item $F\subseteq X$ is \textit{$\R$-unbounded} if no response meets all challenges in $F$.
			\item $F\subseteq Y$ is \textit{$\R$-dominating} if every challenge is met by some response in $F$.
			\item $\R$ is non-trivial if $X$ is $\R$-unbounded and $Y$ is $\R$-dominating. For non-trivial $\R$, define
			\begin{itemize}
				\item $\bb(\R)\coloneq\min\{|F|:F\subseteq X \text{ is }\R\text{-unbounded}\}$, and
				\item $\dd(\R)\coloneq\min\{|F|:F\subseteq Y \text{ is }\R\text{-dominating}\}$.
			\end{itemize}

		\end{itemize}
		
	\end{dfn}
	
	Let us reformulate $\sIstar$ and $\sIwstar$ in the framework of relational systems.
	\begin{dfn}
		Let $x,y\in 2^\omega$.
		\begin{enumerate}
			\item $\zero$ denotes the set of all sequences $z\in2^\omega$ such that $z(n)=0$ for almost all $n$.
			\item $x\rsp y$ if $x(n)=y(n)=1$ and $x(m)=1$ and $y(m)=0$ for infinitely many $n,m<\omega$, i.e., $x^{-1}(\{1\})$ and $y^{-1}(\{1\})$ are infinite and  $x^{-1}(\{1\})$ is split by $y^{-1}(\{1\})$. $y\rp x$ if $\lnot(x\rsp y)$. Note that $y\rp x$ holds whenever $y\in\zero$.  
			
			%		\item Let $j\in2$ and $n<\omega$.
			%		$y\rp_{j,n} x$ if for all $m\geq n$, $x(m)=0$ or $y(m)=1-j$ holds. 
			%		Note $\rp=\bigcup_{j\in2}\bigcup_{n<\omega}\rp_{j,n}$.
			
			\item $y\tri x$ if $x\in 2^\omega\setminus\zero$ and $y\rp x$, i.e., I wins with the play $x$ against the play $y$ of II in the splitting* game. $y\trii x$ if either $x\in 2^\omega\setminus\zero$ and $y\rp x$ or $y\in\zero$, i.e., I wins with the play $x$ against the play $y$ of II in the splitting** game.
			
			\item $\str$ denotes the set of all I's strategies,
			namely, $\str\coloneq2^{(2^{<\omega})}$. For $\sigma\in\str$, $\sigma*y$ denotes the play of I according to the strategy $\sigma$ and the play $y$ of II, namely, $\sigma*y(n)\coloneq\sigma(y\on n)$ for $n<\omega$. 
			%$y\rp \sigma$ denotes $y\rp \sigma*y$ and analogously for $\tri$ and $\trii$.
			$y\tri \sigma$ denotes $y\tri \sigma*y$ and analogously for  $\trii$.
			\item Define relational systems $\rs$ and $\rss$ by $\rs\coloneq\langle 2^\omega,\str,\tri\rangle$ and $\rss\coloneq\langle 2^\omega,\str,\trii\rangle$. 
			
		\end{enumerate}
		Thus, $\sIstar=\bb(\rs)$ and $\sIwstar=\bb(\rss)$ hold.
	\end{dfn}
	
	%Therefore, $y\trii x$ if and only if either $y\tri x$ or $y\in\zero$.

	The two relational systems $\rs$ and $\rss$ behave the same way in most cases: for any $y\in2^\omega$ and $\sigma\in\str$, \textit{whenever $y\notin\zero$,} $y\tri \sigma$ if and only if $y\trii \sigma$.
	Nevertheless, their bounding numbers $\sIstar$ and $\sIwstar$ might be consistently different.
	
	%\[y\trii x\iff
	%\begin{cases}
	%	y\tri x, \text{ or }\\
	%	y\in \zero.
	%\end{cases}
	%\]
	
	%\[y\tri x\iff
	%\begin{cases}
	%	x\notin\zero \text{ and }y\trii x &\text{ if }y\notin\zero\\
	%	x\notin\zero &\text{ if }y\in\zero.
	%\end{cases}
	%\]
	
	%Therefore, this separation result suggests the following:
	%For two relational systems $\R=\langle X,Y,\rel\rangle$, $\R^\prime=\langle X,Y,\rel^\prime\rangle$ on the reals, 
	%even if $\rel = \rel^\prime$ on $(X\setminus C)\times Y$ for some countable $C\subseteq X$, 
	%$\bb(\R)$ and $\bb(\R^\prime)$ might be different.
	
	%Thus, $\rs$ and $\rss$ behave the same way modulo the countable set $\zero$. 
	
	%We explore this phenomenon for other relational systems.
	We investigate whether a similar phenomenon can occur for other relational systems.
	More precisely, we study how the bounding number of $\R=\langle X,Y,\rel\rangle$ is affected by changing the relation $x\rel y$ only when $x$ belongs to some fixed countable set $C\subseteq X$.

	As a test case, we focus on the evasion number $\ee$ since the two relational systems $\rs$ and $\rss$ are similar to the combinatorial concept of evasion and prediction, in that both a strategy and a predictor decide their value by seeing the previous values of a function. % in the sense that a strategy corresponds to a predictor and a play of Player II corresponds to a function being predicted. 
	In what follows, we naturally extend the notation $\zero$ to $\oo$, so that $\zero$ denotes the set of all eventually zero functions in $\oo$.
	
	For several relational systems $\R=\langle X,Y,\rel\rangle$ associated with evasion and prediction, 
	we define the $\zero$-version $\R_\zero$ by switching the prediction relation $x\rel y$ only when $x\in\zero$. We compare the bounding numbers $\bb(\R)$ and $\bb(\R)_\zero\coloneqq\bb(\R_\zero)$.
	
	On the one hand, we show that this ``$\zero$-operation'' %modification %$\zero$-operation 
	changes the standard evasion number $\ee$:
	
	\begin{teorema}[\Cref{thm:Con_eezero<ee}]
		$\ee_\zero<\ee$ is consistent.
	\end{teorema}
	
	On the other hand, we find a variation $\eeg_k$ of the evasion number $\ee$, which is not changed by the $\zero$-operation:
	\begin{teorema}[\Cref{thm:eegk_zero}]
		For any $k\geq1$, $(\eeg_k)_\zero=\eeg_k$.
	\end{teorema}

	%for (the standard) evasion number $\ee$. 
	%
	%
	%On the other hand, we find another example of $\R$, which is a variation of the standard concept of evasion and prediction, such that $\bb(\R)$ stays the same even after such a change.
	%\begin{dfn}
	%	\begin{itemize}
		%		\item $\fstr\coloneq \bigcup_{n<\omega}2^{(2^{<n})}$, the set of all finite partial strategies.
		%		\item $\pstt\coloneq\{(\sigma,F):\sigma\in\fstr, F\in[2^\omega]^{<\omega}\}$. $(\sigma^\prime, F^\prime)\leq(\sigma,F):\Leftrightarrow \sigma^\prime\supseteq\sigma, F^\prime\supseteq F$ and for all $n\in\left[|\sigma|,|\sigma^\prime|\right)$ and $y\in F$, $\sigma^\prime(y\on n)\leq y(n)$.
		%		\item $\pst\coloneq\{(\sigma,F)\in\pstt: F\subseteq 2^\omega\setminus\zero \}$ and the order is defined by restriction.
		%		\item For both $\pst$ and $\pstt$, $\sigma_G$ denotes the generic strategy $\sigma_G\coloneq\bigcup_{(\sigma,F)\in G}\sigma$ for a generic filter $G$. 
		%		
		%	\end{itemize}
	%	
	%\end{dfn}
	%
	%\begin{lem}
	%	\begin{enumerate}
		%		\item $\pst$ and $\pstt$ are $\sigma$-centered.
		%		\item For $y\in 2^\omega$, $\Vdash_{\pstt} y\rp_0\sigma_G$, where $\rp_0\coloneq\bigcup_{n<\omega}\rp_{0,n}$.
		%		\item For $y\in 2^\omega\setminus\zero$, $\Vdash_{\pst} y\rp_0\sigma_G$.
		%		
		%		
		%		
		%		\item For $y\in 2^\omega\setminus\zero$, $\Vdash_{\pstt} \sigma_G*y\in2^\omega\setminus\zero$.
		%		\item For $y\in 2^\omega$, $\Vdash_{\pst} \sigma_G*y\in2^\omega\setminus\zero$.
		%	\end{enumerate}
	%\end{lem}
	%\section{Preliminaries}
	%\input{zero_RS.tex}
	\section{$\zero$-prediction}
	%\subsection{$\zero$-prediction}
	First we clarify the notation of the standard evasion and prediction.
	\begin{dfn}
		\begin{itemize}
			\item A pair $\pi=(D,\{\pi_n:n\in D\})$ is a \textit{predictor} if $D\in\ooo$ and $\pi_n\colon\omega^n\to\omega$ for each $n\in D$.
			\item For a predictor $\pi=(D,\{\pi_n:n\in D\})$ and a function $f\in\oo$, $\pi$ \textit{predicts} $f$ if $f(n)=\pi_n(f\on n)$ for almost all $n\in D$.
			\item $\Pred$ denotes the set of all predictors.
			\item $\mathbf{PR}\coloneq\langle \oo,\Pred,\sqsubset^\mathrm{p}\rangle$, where $f\sqsubset^\mathrm{p}\pi\colon\Leftrightarrow f$ is predicted by $\pi$.
		\end{itemize}

	\end{dfn}
	
	We introduce \textit{$\zero$-prediction}. 
	\begin{dfn}
		Let $\pi$ be a predictor and $f\in\oo$. 
		\begin{enumerate}
			\item $\pi$ \textit{$\zero$-predicts} $f$, denoted by $f\zp\pi$, if either:
			\begin{itemize}
				\item $f\notin\zero$ and $f\sqsubset^\mathrm{p}\pi$, or%$\pi$ predicts $f$, or
				\item $f\in\zero$ and $\lnot(f\sqsubset^\mathrm{p}\pi)$.%$\pi$ does not predict $f$.
			\end{itemize}
			\item Put $\mathbf{PR}_\zero\coloneq\langle \oo, \Pred,\zp \rangle$ and $\ee_\zero\coloneq\bb(\mathbf{PR}_\zero)$.
		\end{enumerate}
	\end{dfn}
	
	\begin{lem}\label{lem:ee_zero_ee}
		$\ee_\zero\leq\ee$.
	\end{lem}
	\begin{proof}
		By identifying $\omega$ and $\omega\setminus1$, we can take an evading family  $F\subseteq(\omega\setminus1)^\omega$ of size $\ee$.
		Then, $F$ is also a $\zero$-evading family.
	\end{proof}

	%\begin{lem}
	%	$\ee_\zero\leq\none$.
	%\end{lem}
	%\begin{proof}
	%	For a predictor $\pi$, $\{f\in\oo:f\zp\pi\}\subseteq\{f\in\oo\setminus\zero:f\sqsubset^\mathrm{p}\pi\}\cup\zero\in\mathcal{E}$. 
	%\end{proof}

	%\begin{cor}
	%	\label{cor_p_ez}
	%	$\pp=\mm(\sigma\text{-centered})\leq\ee_\zero$.
	%\end{cor} weaker than below
	
	%For Recall that $\bar{I}=\langle I_n:n<\omega\rangle$ is an interval partition of $\omega$ if  it is a partition and $I_n$ is a (non-empty) interval for each $n$. Assume the intervals are arranged in ascending order, i.e. $\min I_0 = 0, \max I_n + 1 = \min I_{n+1}$ for every $n$.
	%Let $\mathsf{IP}$ be the set of all interval partitions.
	For interval partitions $I=\langle I_n:n<\omega\rangle$ and $J=\langle J_m:m<\omega\rangle$ of $\omega$, we say $I$ \textit{is dominated by} $J$, denoted by $I\leq^*J$, if 
	%$ (\forall^\infty m)(\exists n)$
	for almost all $m$ there is $n$ such that $I_n \subseteq J_m$. It is known that $\bb(\langle\mathsf{IP}, \mathsf{IP}, \leq^*\rangle) = \bb$ where $\mathsf{IP}$ denotes the set of all interval partitions of $\omega$ (see \cite{Bla10}).
	\begin{lem}
		\label{lem:ezero_min_eb}
		$\ee_\zero\geq\min\{\ee,\bb\}$.
	\end{lem}
	\begin{proof}
		Let $F\subseteq\oo$ with size $<\min\{\ee,\bb\}$ and $F^-\coloneq F\setminus\zero$.
		For $f\in F^-$ and $k<\omega$, let $i^f_k$ denote the $k$-th element of $\{i<\omega:f(i)\neq0\}\in\ooo$ and put $I^f_k\coloneq\left[i_{k-1},i_k\right)$ ($i_{-1}\coloneq0)$.
		Define an interval partition $I^f$ by $I^f\coloneq\langle I^f_k:k<\omega\rangle$.
		Since $|F^-|<\bb$, there is an interval partition $J=\langle J_m:m<\omega\rangle$ dominating all $\{I^f:f\in F^-\}$. %, namely, for any $f\in F^- \text{ and for all but finitely many }m<\omega,\text{ there is }k<\omega\text{ such that }I^f_k\subseteq J_m.$
		Since $|F^-|<\ee$, there is a predictor $\pi=(D,\langle \pi_k:k\in D\rangle)$ predicting all $f\in F^-$.
		Inductively construct $A\coloneq\{m_0<m_1<\cdots\}\in\ooo$ and $D^\prime\coloneq\{d_0<d_1<\cdots\}\in[D]^\omega$ such that $\max J_{m_i}<d_i<\min J_{m_{i+1}}$ holds for all $i<\omega$.
		Define a predictor $\pi^\prime\coloneq(D^\prime,\langle \pi^\prime_k:k\in D^\prime\rangle)$ as follows: for $k=d_i\in D^\prime$ and $\sigma\in\omega^k$,
		\begin{equation}
			\label{eq_piPrime}
			\pi^\prime_k(\sigma)\coloneq
			\begin{cases}
				\pi_k(\sigma) & \text{if } \sigma(n)\neq0\text{ for some }n\in J_{m_i},\\
				1 & \text{otherwise. }
			\end{cases}
		\end{equation}
		Clearly $\pi^\prime$ does not predict any $f\in\zero$, so it suffices to show that $\pi^\prime$ predicts all $f\in F^-$.
		Since $J$ dominates $I^f$, for all but finitely many $i<\omega$ there is $k_i\geq 1$ such that $I_{k_i}^f\subseteq J_{m_i}$, so particularly $n\coloneq\min I_{k_i}^f\in J_{m_i}$ satisfies $f(n)\neq0$.
		Thus, by \eqref{eq_piPrime} $\pi^\prime_{d_i}(f\on d_i)=\pi_{d_i}(f\on d_i)$ for all but finitely many $i$.
		Therefore, since $\pi$ predicts $f$, $\pi^\prime$ also predicts $f$.
	\end{proof}

	\begin{rem}
		A similar inequality holds for splitting game numbers: $\min\{\sIwstar,\bb\}\leq\sIstar\leq\sIwstar$ was proved in \cite{CGHY26}.
	\end{rem}

	%\subsection{no-influence case}
	
	We will see that $\ee_\zero$ and $\ee$ are consistently different in the next section. 
	%This result should be compared to the following case
	In contrast, there is a case where this ``$\zero$-operation'' does not change the original cardinal invariant. 
	
	Blass studied many variations of the evasion number in \cite[Chapter 10]{Bla10}.
	Among such variations, we treat\textit{ global (adaptive) prediction of width $k$} for a natural number $k\geq1$. This concept differs from standard prediction in the following two respects: %The difference from the standard prediction is the following two points:
	First, for a 
	predictor $\pi=(D,\{\pi_n:n\in D\})$, we always require $D=\omega$.
	Second, 
	%the predictor guesses $k$-many possibilities of the $f(n)$ of a function $f$ being predicted, not necessarily the exact value of $f(n)$.
	at each stage $n$, the predictor guesses $k$ possible values for $f(n)$, rather than necessarily guessing the exact value $f(n)$.
	This prediction is formalized (more concisely) as follows: 
	\begin{dfn}
		Let $k\geq 1$.
		\begin{itemize}
			\item A function $\pi$ is a \textit{global predictor of width $k$} if $\dom(\pi)=\fsq$ and $\ran(\pi)\subseteq[\omega]^k$.
			\item $\Pred_k$ denotes the set of all global predictors of width $k$.
			\item For $\pi\in\Pred_k$ and a function $f\in\oo$, we say $\pi$ \textit{globally predicts} $f$, denoted by $f\in^\mathrm{p}\pi$, if $f(n)\in\pi(f\on n)$ for almost all $n\in\omega$.
			\item Put $\mathbf{PR}^\mathrm{G}_k\coloneq\langle \oo,\Pred_k,\in^\mathrm{p}\rangle$ and  $\eeg_k\coloneqq\bb(\mathbf{PR}^\mathrm{G}_k)$.
		\end{itemize}
	\end{dfn}
	%Blass showed $\eeg_1=\aleph_1$, so hereafter we assume $k\geq2$.

	Recall that for a natural number $k\geq2$, a poset $\p$ is \textit{$\sigma$-$k$-linked} if there is a countable decomposition $\p=\bigcup_{n\in\omega}Q_n$ such that for any $n$, any $k$ conditions in $Q_n$ have a common extension in $\p$. $\mm(\sigma\text{-}k\text{-linked})$ denotes the least cardinality $\kappa$ such that Martin's axiom for $\sigma$-$k$-linked posets fails at $\kappa$. 
	
	Blass showed the following:
	
	\begin{lem}[{\cite{Bla10}}]
		$\eeg_1=\aleph_1$ and  $\mm(\sigma\text{-}k\text{-linked})\leq\eeg_k\leq\addn$ for any $k\geq2$.
	\end{lem}
	
	We introduce $\zero$-global prediction of width $k$:
	\begin{dfn}
		Let $k\geq 1$.
		\begin{itemize}
			\item For a function $f\in\oo$ and a global predictor $\pi$, we say $\pi$ \textit{$\zero$-globally predicts} $f$, denoted by $f\in^\mathrm{p}_\zero\pi$, if either: 
			\begin{itemize}
				\item $f\notin\zero$ and $f\in^\mathrm{p}\pi$, or%$\pi$ predicts $f$, or
				\item $f\in\zero$ and $\lnot(f\in^\mathrm{p}\pi)$.%$\pi$ does not predict $f$.
			\end{itemize}
			\item Put $(\mathbf{PR}^\mathrm{G}_k)_\zero\coloneq\langle \oo,\Pred_k,\in^\mathrm{p}_\zero\rangle$ and $(\eeg_k)_\zero\coloneqq\bb((\mathbf{PR}^\mathrm{G}_k)_{\zero})$.
		\end{itemize}

	\end{dfn}
	
	The $\zero$-operation does not change the value of $\eeg_k$:
	\begin{thm}\label{thm:eegk_zero}
		For any $k\geq1$, $(\eeg_k)_\zero=\eeg_k$.
	\end{thm}
	%(This is essentially because $\eeg_k=\min\{\bb,\eeg_k\}\leq(\eeg_k)_\zero\leq\eeg_k$.)
	\begin{proof}
		To see $(\eeg_k)_\zero\leq\eeg_k$, take a $\mathbf{PR}^\mathrm{G}_k$-unbounded family $F$ of size $\eeg_k$ in $(\omega\setminus1)^\omega$. $F$ is also  $(\mathbf{PR}^\mathrm{G}_k)_{\zero}$-unbounded, which implies $(\eeg_k)_\zero\leq\eeg_k$.
		
		To see $\eeg_k\leq (\eeg_k)_\zero$, 
		let $F\subseteq\oo$ of size $<\eeg_k$ and $F^-\coloneq F\setminus\zero$.
		For $f\in F^-$ and $l<\omega$, let $i^f_l$ denote the $l$-th element of $\{i<\omega:f(i)\neq0\}\in\ooo$, and put $I^f_l\coloneq\left[i_{l-1},i_l\right)$ ($i_{-1}\coloneq0)$.
		Define an interval partition $I^f$ by $I^f\coloneq\langle I^f_l:l<\omega\rangle$.
		Since $|F^-|<\eeg_k\leq\addn\leq\bb$, there is an interval partition $J=\langle J_m:m<\omega\rangle$ dominating all $\{I^f:f\in F^-\}$.
		%, namely, for any $f\in F^- \text{ and for all but finitely many }m<\omega, \text{ there is }k<\omega\text{ such that }I^f_k\subseteq J_m.$
		Since $|F^-|<\eeg_k$, there is $\pi\in\Pred_k$ that globally predicts all $f\in F^-$.
		%Pick some $A\coloneq\{m_0<m_1<\cdots\}\in\oo$ and $D^\prime\coloneq\{d_0<d_1<\cdots\}\in[D]^\omega$ such that $\max J_{m_i}<d_i<\min J_{m_{i+1}}$ holds for all $i<\omega$.
		Define $\pi^\prime\in\Pred_k$ as follows: for $\sigma\in\omega^n$ and $m\in\omega$ with $n\in J_m$:
		\begin{equation*}
			\label{eq_piPrime_k}
			\pi^\prime(\sigma)\coloneq
			\begin{cases}
				\pi(\sigma) & \text{if } \sigma(i)\neq0\text{ for some }i\in J_{m-1},\\ %~(\text{put }J_{m-1}\coloneqq\emptyset),\\
				\{1,\ldots,k\} & \text{otherwise. }
			\end{cases}
		\end{equation*}
		Then $\pi^\prime$ globally predicts every $f\in F^-$ and does not globally predict any $f\in\zero$.
		Thus, $F$ is $(\mathbf{PR}^\mathrm{G}_k)_{\zero}$-dominated by $\pi^\prime$ and hence $|F|<(\eeg_k)_\zero$.
		Therefore, we have $\eeg_k\leq (\eeg_k)_\zero$.
	\end{proof}

	\section{Consistency of $\ee_\zero<\ee$}
	Towards the consistency of $\ee_\zero<\ee$, we introduce a poset $\prp$ which adds a generic predictor $\pi_G$ and hence increases $\ee$ by iteration. 
	\begin{dfn}
		\label{dfn_PR}
		\textit{	Prediction forcing }$\prp$ consists of tuples $(d,\pi,F)$ satisfying:
		\begin{enumerate}
			\item $d\in\sq$.
			\item $\pi=\langle\pi_n:n\in d^{-1}(\{1\})\rangle$.
			\item for each $n\in d^{-1}(\{1\})$, $\pi_n$ is a finite partial function of $\omega^n\to \omega$.
			\item $F\in[\oo]^{<\omega}$.
			\item for each $f,f^\prime\in F, f\on|d|=f^\prime\on|d|$ implies $f=f^\prime$.
		\end{enumerate}
		$(d^\prime,\pi^\prime,F^\prime)\leq(d,\pi,F)$ if:
		\begin{enumerate}[(i)]

			\item $d^\prime\supseteq d$.
			\item $\forall n\in d^{-1}(\{1\}), \pi_n^\prime\supseteq\pi_n$.
			\item $F^\prime\supseteq F$.
			\item \label{item_PR_order_long}
			For all $ n\in(d^\prime)^{-1}(\{1\})\setminus d^{-1}(\{1\})$ and $f\in F$, we have $f\on n\in\dom(\pi^\prime_n)$ and $\pi^\prime_n(f\on n)=f(n)$.
			
		\end{enumerate}
		For a generic filter $G$, define the generic predictor $\pi_G=(D_G, \langle\pi^G_n:n\in D_G\rangle)$ by $D_G=\bigcup\{d^{-1}(\{1\}):(d,\pi,F)\in G\}$ and for $n\in D_G$,
		$\pi^G_n\coloneqq\bigcup\{\pi_n:(d,\pi,F)\in G,\ n\in d^{-1}(\{1\})\}$. %$\pi^G_n=\pi_n$ for some $(d,\pi,F)\in G$. This is well-defined by the definition of the order.
	\end{dfn}
	It is easy to see that $\prp$ is $\sigma$-centered, and hence ccc.

	We can obtain a forcing which increases $\ee_\zero$ by just restricting $\prp$:
	\begin{dfn}
		$\prp_\zero\coloneq\{(d,\pi,F)\in\prp:F\subseteq\oo\setminus\zero\}$ and the order is defined by restriction of $\prp$. $\pi_G$ is defined in the same way.
	\end{dfn}
	$\prp_\zero$ is $\sigma$-centered as well and hence it is ccc. 
	\begin{lem}
		%	\begin{enumerate}
			%		\item $\prp_\zero$ is $\sigma$-centered.
			%		\item %where $\dot{\pi}_G$ denotes the generic predictor. 
			%	\end{enumerate}
		For $f\in\oo$, $\Vdash_{\prp_\zero} f\zp\dot{\pi}_G$. 
	\end{lem}
	
	\begin{proof}
		If $f\notin\zero$,  we obtain $\Vdash_{\prp_\zero} f\sqsubset^\mathrm{p}\dot{\pi}_G$ by the same proof as the standard $\prp$.
		Assume $f\in\zero$ and $(d,\pi,F)\in\prp_\zero$ and $m<\omega$ are given.
		We shall find $n\geq m$ and $(d^\prime,\pi^\prime,F)\leq(d,\pi,F)$ forcing that $n\in D_G$ and $\pi_G(f\on n)\neq f(n)$.
		Since $f\notin F$, we may assume that $f\on m\neq g\on m$ for all $g\in F$ and $f(n)=0$ for all $n\geq m$ by increasing $m$, and that $|d|\geq m$ by adding enough $0$'s after the binary sequence $d$.
		Let $n\coloneq|d|$ and $d^\prime\coloneq d^\frown 1\in2^{n+1}$.
		Define the finite partial function $\pi^\prime_n\colon \omega^n\to\omega$ by $\dom(\pi^\prime_n)\coloneq\{g\on n:g\in F\cup\{f\}\}$ and for $g\on n\in\dom(\pi^\prime_n)$,
		\begin{equation}
			\label{eq_pi}
			\pi^\prime_n(g\on n)\coloneq
			\begin{cases}
				g(n) & \text{if } g\in F,\\
				1 & \text{if }g=f.
			\end{cases}
		\end{equation}
		This is well-defined since for any $g,g^\prime\in F\cup\{f\}$, $g\on n=g^\prime\on n$ implies $g=g^\prime$.
		Let $\pi^\prime\coloneq \pi^\frown \pi^\prime_n$ be the sequence of finite partial functions indexed by $k\in (d^\prime)^{-1}(\{1\})$. 
		By \eqref{eq_pi}, $(d^\prime,\pi^\prime,F)$ extends $(d,\pi,F)$ and forces that $n\in (d^\prime)^{-1}(\{1\})\subseteq D_G$ and $\pi^G_n(f\on n)=\pi^\prime_n(f\on n)=1\neq 0=f(n)$.
	\end{proof}

	When showing the consistency of $\ee_\zero<\ee$, we will use the \textit{Fr-linkedness} property to keep $\ee_\zero$ small in the forcing iteration. %This method stems from Miller's work \cite{Mil81} and 
	Fr-linkedness was formalized by Mej{\'\i}a in \cite{Mej19} and the formalization was described in more detail in \cite{CGHY26}. 
	%To force $\ee_\zero<\ee$, we need a more detailed formulation, so 
	Let us review the framework of Fr-linkedness in the following.
	
	\begin{dfn}
		Let $\p$ be a poset.
		\begin{enumerate}
			\item For a countable sequence $\bar{p}=\langle p_m:m<\omega\rangle\in\p^\omega$, we define $\dot{W}(\bar{p})$ as the $\p$-name of an index set of the sequence $\bar{p}$ as follows:
			\[
			\Vdash_\p\dot{W}(\bar{p})\coloneq\{m<\omega:p_m\in\dot{G}\}
			\]
			where $\dot{G}$ denotes the canonical $\p$-name of a generic filter.
			\item $Q\subseteq \p$ is \emph{Fr-linked} if there exists a function $\lim\colon Q^\omega\to\p$ such that for any countable sequence 
			$\bar{q}\in Q^\omega$, 
			\begin{equation*}
				\textstyle{\lim\bar{q}} \Vdash |\dot{W}(\bar{q})|=\omega.
			\end{equation*}
			%Additionally, if $\ran(\lim)\subseteq Q$, we say $Q$ is Fr-linked.
			\item For an infinite cardinal $\mu$, $\p$ is \emph{$\mu$-Fr-linked} if it is a union of $\mu$-many Fr-linked components. When $\mu=\aleph_0$, we use $\sigma$ instead as usual. Define $<\mu$-Fr-linkedness in the same way (for uncountable $\mu$). 
		\end{enumerate}
		%	We often say ``$\p$ has Fr-limits'' instead of ``$\p$ is $\sigma$-Fr-linked''.
	\end{dfn}
	Since any singleton $\{q\}\subseteq\p$ is Fr-linked, we have:
	\begin{lem}
		\label{lem:poset_is_its_size_Frlinked}
		Every poset $\p$ is $|\p|$-Fr-linked. In particular, Cohen forcing $\mathbb{C}$ is $\sigma$-Fr-linked.
	\end{lem}
	
	Brendle and Shelah (essentially) proved the following:
	\begin{lem}[{\cite{BS_E_and_P_2}}]
		\label{lem_prp_is_Frlinked}
		The prediction forcing $\prp$ is $\sigma$-Fr-linked.
	\end{lem}
	
	%	We will use Fr-linkedness to control the values of $\non(M_{\finfin})$ and $\cov(M_{\finfin})$, not by goodness properties but by directly using Fr-limits (\Cref{lem:closed_Fr_limits_keep_nonMfinfin_small}). To this end, we formulate a finite support iteration of ${<}\kappa$-Fr-linked forcings. The formalization is based on \cite{Mej19}, \cite{CGHY24}.
	
	We formulate a finite support iteration (fsi) of $<\kappa$-Fr-linked forcings:
	\begin{dfn}
		\label{dfn:Gamma_iteration}
		Let $\kappa$ be an uncountable regular cardinal.
		\begin{itemize}
			\item 
			A finite support iteration $\p_\gamma=\langle(\p_\xi,\qd_\xi):\xi<\gamma\rangle $ of ccc forcings is a \emph{${<}\kappa$-Fr-iteration with witnesses $\langle\theta_\xi:\xi<\gamma\rangle$ and $\langle\dot{Q}_{\xi,\zeta}:\zeta<\theta_\xi,\xi<\gamma\rangle$} if for any $\xi<\gamma$, $\theta_\xi$ is a cardinal $<\kappa$ and
			$\langle\dot{Q}_{\xi,\zeta}:\zeta<\theta_\xi\rangle$ are $\p_\xi$-names satisfying: 
			\[
			\Vdash_{\p_\xi}\dot{Q}_{\xi,\zeta}\subseteq\qd_\xi\text{ is Fr-linked for }\zeta<\theta_\xi\text{ and }\bigcup_{\zeta<\theta_\xi}\dot{Q}_{\xi,\zeta}=\qd_\xi.
			\]
			\item A condition $p\in\p_\gamma$ is \emph{determined} if for each $\xi\in\dom(p)$, there is $\zeta_\xi<\theta_\xi$ such that  $\Vdash_{\p_\xi} p(\xi)\in \dot{Q}_{\xi,\zeta_\xi}$. Note that there are densely many determined conditions (proved by induction on $\gamma$).
			
		\end{itemize}
	\end{dfn}
	
	Fr-limits for conditions of the iteration $\p_\gamma$ are defined for ``refined'' sequences:
	
	\begin{dfn}
		\label{dfn:uniform_delta_system}
		Let $\p_\gamma$ be a ${<}\kappa$-Fr-iteration and $\delta$ be an ordinal.
		We say that $\bar{p}=\langle p_m: m<\delta\rangle\in(\p_\gamma)^\delta$ is \emph{a uniform $\Delta$-system} if:
		\begin{enumerate}
			\item Each $p_m$ is determined, witnessed by $\langle\zeta^m_{\xi} :\xi\in\dom(p_m)\rangle$,
			\item the family $\{\dom(p_m): m<\delta\}$ is a $\Delta$-system with root $\nabla$,
			\item there is a sequence $\langle\zeta_\xi^{*} :\xi\in\nabla\rangle$ such that for $\xi\in\nabla$, $\zeta_\xi^{*}=\zeta^m_\xi$ for all $m<\delta$, i.e., all $p_m(\xi)$ are forced to be in a common Fr-linked component for $\xi\in\nabla$,
			\item all $\dom(p_ m)$ have $n^\prime$ elements, and $\dom(p_ m)=\{\a_{n, m}:n<n^\prime\}$ is the increasing enumeration,
			\item there is $r^\prime\subseteq n^\prime$ such that $n\in r^\prime\Leftrightarrow\a_{n, m}\in\nabla$ for $n<n^\prime$,
			\item for $n\in n^\prime\setminus r^\prime$, $\langle\a_{n, m}: m<\delta \rangle$ is (strictly) increasing.
		\end{enumerate}
	\end{dfn}
	
	\begin{dfn} \label{dfn:of_lim_ite}
		Let $\p_\gamma$ be a ${<}\kappa$-Fr-iteration and $\bar{p}=\langle p_m: m<\omega\rangle\in(\p_\gamma)^\omega$ be a uniform $\Delta$-system with root $\nabla$. We (inductively) define $p^\infty=\lim\bar{p}$ as follows:
		\begin{enumerate}
			\item $\dom(p^\infty)\coloneq\nabla$,
			\item $p^\infty\on\xi\Vdash_{\p_\xi} p^\infty(\xi)\coloneq\lim\langle p_m(\xi):m\in\dot{W}(\bar{p}\on\xi)\rangle$ for $\xi\in\nabla$, where $\bar{p}\on\xi\coloneq\langle p_m\on\xi:m<\omega\rangle\in(\p_\xi)^\omega$. \label{eq_lim_second}
		\end{enumerate}
	\end{dfn}
	
	To see that the second item is valid, $\dot{W}(\bar{p}\on\xi)$ has to be infinite, which is true: 
	
	\begin{lem}[{\cite{Mej19}, \cite[Lemma 3.6]{CGHY26}}]\label{lem:Fr-limit_principle}
		Let $\p_\gamma$ be a ${<}\kappa$-Fr-iteration. Let $\bar{p}=\langle p_m: m<\omega\rangle\in(\p_\gamma)^\omega$ be a uniform $\Delta$-system and $p^\infty\coloneq\lim\bar{p}$.
		Then
		$p^\infty\Vdash_{\p_\gamma}\lvert\dot{W}(\bar{p})\rvert=\omega$.
	\end{lem}
	To show that Fr-linkedness helps to keep $\ee_\zero$ small, we characterize $\ee_\zero$ with the following relational system:  
	\begin{dfn}
		$\Pred_\zero\coloneq\{\pi\in \Pred: \pi \text{ does not predict any } f\in\zero\}$.
		$\mathbf{PR}^\prime_\zero\coloneq\langle \oo,\Pred_\zero,\zp \rangle$.%$\ee_\zero\coloneq\bb(\mathbf{PR}_\zero)$.
	\end{dfn}
	\begin{lem}
		\label{lem_b_of_przeroprime}
		$\bb(\mathbf{PR}^\prime_\zero)=\ee_\zero$.
	\end{lem}
	\begin{proof}
		$\bb(\mathbf{PR}^\prime_\zero)\leq\ee_\zero$ is clear. To see $\ee_\zero\leq \bb(\mathbf{PR}^\prime_\zero)$, let $F\subseteq\oo$ of size $<\ee_\zero$.
		By Lemma \ref{lem:ezero_min_eb}, we particularly know that $\ee_\zero$ is uncountable. Thus $F^\prime\coloneq F\cup\zero$ has size $<\ee_\zero$, so some $\pi\in \Pred$ $\zero$-predicts all $f\in F^\prime$. This $\pi$ has to be in $\Pred_\zero$, so $\pi$ witnesses that $F$ is $\mathbf{PR}^\prime_\zero$-bounded.
	\end{proof}
	
	We use the following notation.
	\begin{dfn}
		For a predictor $\pi=(D,\{\pi_n:n\in D\})$, $f\in\oo$ and $n<\omega$, $f\sqsubset^\mathrm{p}_n\pi$ denotes that $f(k)=\pi_k(f\on k)$ for all $k\geq n$ in $D$. 
	\end{dfn}

	%\begin{lem}
	%	\label{lem_ez_small}
	%	Let $\p_\gamma$ be a $<\kappa$-$\Lambda_\mathrm{Fr}$-iteration.
	%	Then, $\Vdash_{\p_\gamma} C_{[\gamma]^{<\kappa}}\lq\mathbf{PR}^\prime_\zero$, particularly $\ee_\zero\leq\kappa$, witnessed by the first $\kappa$-many Cohen reals.
	%\end{lem}
	
	The Fr-linkedness property and Cohen reals help to keep $\ee_\zero$ small through the forcing iteration: 
	\begin{lem}\label{lem_ez_small}
		Let $\kappa$ be an uncountable regular cardinal, $\gamma>\kappa$ be a limit ordinal and $\p_\gamma$ be a $<\kappa$-Fr-iteration whose first $\kappa$-many iterands are Cohen forcings $\mathbb{C}$.
		Then, $\Vdash_{\p_\gamma}\ee_\zero\leq\kappa$.
	\end{lem}
	
	\begin{proof}
		We show that the first $\kappa$-many Cohen reals $\langle\dot{c}_\a:\a<\kappa\rangle$ (as members of  $\oo$) witness $\Vdash_{\p_\gamma} \ee_\zero\leq\kappa$, by focusing on the characterization $\ee_\zero=\bb(\mathbf{PR}^\prime_\zero)$ in Lemma \ref{lem_b_of_przeroprime}. 
		
		Assume towards contradiction that there exist a condition $p\in\p_\gamma$ and a $\p_\gamma$-name  $\dot{\pi}=(\dot{D},\langle\dot{\pi}_n:n\in\dot{D}\rangle)$ of a member of $\Pred_\zero$ such that for all $\a<\kappa$,
		$p\Vdash \dot{c}_\a\zp \dot{\pi}$.
		Since Cohen reals are not in $\zero$, we actually have $p\Vdash \dot{c}_\a\pdtd\dot{\pi}$.
		Let us witness the starting point of the prediction $\pdtd=\bigcup_{n<\omega}\pdtd_n $.
		That is,
		for each $\a<\kappa$ we pick $p_\a\leq p$ and $n_\a<\omega$ such that $p_\a\Vdash \dot{c}_{\a}\pdtd_{n_\a}\dot{\pi}$.
		By extending and thinning, we may assume there is $I\in[\kappa]^{\kappa}$ such that: 
		
		\begin{enumerate}
			\item $\a\in\dom(p_\a)$  for $\a\in I$. %(By extending $p_i$.)
			%\item All $ p_i$ follow a common guardrail $h\in H$. ($|H|<\kappa$.)
			\item $\{p_\a:\a\in I\}$ forms a uniform $\Delta$-system with root $\nabla $. %(By Lemma \ref{lem_uniform_Delta_System_Lemma}.)
			%\item $\b_i\notin\nabla$, hence all $\b_i$ are distinct. %($\b_i$ are unbounded in $\kappa$, so $\b_i$ are eventually out of the finite set $\nabla$.)
			\item All $n_\a$ are equal to $ n^*$.
			\item All $p_\a(\a)$ are the same Cohen condition $s\in\omega^{<\omega}$. 
			\item $|s|=n^*$. (By extending $s$ or increasing $n^*$.)
		\end{enumerate}
		
		In particular, we have that:
		\begin{equation}
			\label{eq_prop_of_refined_pi}
			\text{For each }\a\in I, ~p_\a\text{ forces }\dot{c}_{\a}\on n^*=s\text{ and }\dot{c}_{\a}\pdtd_{n^*}\dot{\pi}.
		\end{equation}
		
		Pick some countable $\{\a_0<\a_1<\cdots\}\in[I\setminus\nabla]^\omega$. For $m<\omega$, define $q_m\leq p_{\a_m}$ by extending the $\a_m$-th position to $q_m(\a_m)\coloneqq s^\frown\zero_m\in \omega^{n^*+m}$ where $\zero_m$ denotes the sequence of length $m$ whose values are all $0$.
		
		By \eqref{eq_prop_of_refined_pi}, we obtain:
		\begin{equation}
			\label{eq_prop_of_qi}
			q_m\Vdash\text{for all  }n\in \left[n^*,n^*+m\right)\cap\dot{D},~\dot{\pi}_n(s^\frown\zero_{n-n^*})=0.
		\end{equation}
		Since $\a_m\notin\nabla$, $\bar{q}\coloneq\langle q_m\rangle_{m<\omega}$ also forms a $\Delta$-system (with root $\nabla$) and hence we can take the limit $q^\infty\coloneq\lim\bar{q}$. Let $f\coloneq s^\frown00\cdots\in\zero$.
		Since $q^\infty\Vdash\exists^\infty m<\omega~q_m\in\dot{G}$ and by \eqref{eq_prop_of_qi},
		\begin{equation*}
			%\label{eq_prop_of_qi}
			q^\infty\Vdash\text{for all  }n\geq n^*\text{ in }\dot{D},~\dot{\pi}_n(f\on n)=0=f(n),
		\end{equation*}
		which contradicts $\dot{\pi}\in \Pred_\zero$.	
	\end{proof}
	
	Now we are ready to prove the consistency of $\ee_\zero<\ee$.
	\begin{thm}
		\label{thm:Con_eezero<ee}
		Let $\kappa\leq\lambda$ be uncountable regular cardinals. Then, it is consistent that $\kappa=\ee_\zero\leq\ee=\lambda$ holds.
	\end{thm}

	%\begin{proof}
	%	Iterate $\mathbb{PR}$, which is $\sigma$-Fr-linked and increases $\ee$.
	%\end{proof}
	
	\begin{proof}
		Let $\mu>\lambda$ be such that $\mu^{<\kappa}=\mu$ and $\cf(\mu)=\lambda$ (e.g., construct the continuous sequence $\langle \mu_\a:\a\leq\lambda\rangle$ by $\mu_0\coloneq\lambda$ and $\mu_{\a+1}\coloneq(\mu_\a^{<\kappa})^+$ and put $\mu\coloneq\mu_\lambda$).
		Let $\p=\p_\gamma$ be a $<\kappa$-Fr-iteration of length $\gamma=\mu$ whose first $\kappa$-many iterands are Cohen forcings and each of the rest is either $\prp$ or a subforcing of $\prp_\zero$ of size $<\kappa$ (by bookkeeping), which is possible by Lemma \ref{lem_prp_is_Frlinked} (and Lemma \ref{lem:poset_is_its_size_Frlinked}). 
		It is easy to see that $\p$ forces $\ee\geq\lambda$ by $\cf(\mu)=\lambda$ and $\ee_\zero\geq\kappa$ by bookkeeping (and $\mu^{<\kappa}=\mu$). 
		Since we are performing a finite support iteration, $\lambda$ many Cohen reals are cofinally added and they witness $\nonm\leq\lambda$. 
		Since $\ee\leq\nonm$ holds in ZFC, we have $\Vdash_\p\lambda\leq\ee\leq\nonm\leq\lambda$ and hence  $\Vdash_\p\ee=\lambda$.
		Lemma \ref{lem_ez_small} implies $\Vdash_\p \ee_\zero\leq\kappa$. Therefore, $\Vdash_{\p}\ee_\zero=\kappa\text{ and }\ee=\lambda$.
	\end{proof}
	
	%	\subsection{detailed analysis}
	
	%	\input{diff}
	%	\input{e2}
	\section{Questions}
	The first question is about the exact value of $\ee_\zero$. By \Cref{lem:ee_zero_ee}, \ref{lem:ezero_min_eb}, we know $\min\{\ee,\bb\}\leq\ee_\zero\leq\ee$. Then we ask:
	\begin{ques}
		Does $\ee_\zero=\min\{\ee,\bb\}$ hold in ZFC? Equivalently, is $\ee_\zero\leq\bb$ true?
	\end{ques} 
	
	The second question is more exploratory.
	The invariants $\sIstar$ and $\sIwstar$ are quite naturally defined: 
	%It is therefore natural to ask whether the same mechanism can be realized in a more natural setting.
	%The splitting* game and the splitting** game are reasonable ways to realize the combination of splitting and infinite games
	The splitting* game and the splitting** game are reasonable game-theoretic formulations of the splitting property. Moreover, the need to distinguish the two games is inevitable, since if both players play finite sets, there is no fair criterion for deciding the winner. 
	
	In contrast, our definition of $\zero$-prediction is somewhat artificial. This leads to the following question:
	\begin{ques}
		Is there another separation result based on the same mechanism as $\sIstar<\sIwstar$, but arising from
		a more natural framework than $\zero$-prediction?
	\end{ques}
	
	%\cite[Question E]{CGHY26}
	%Is there another separation result based on the same mechanism of $\sIstar<\sIwstar$, which seems more natural than $\ee_\zero<\ee$?
	\begin{acknowledgements}
		This article was written for the proceedings of the RIMS Set Theory Workshop 2025 \textit{Recent Developments in Axiomatic Set Theory}, held at Kyoto University RIMS. The author is grateful to the organizer Teruyuki Yorioka for letting him give a talk at the Workshop and submit an article to the proceedings.
		%The author also thanks his supervisor J\"{o}rg Brendle for his helpful comments. 
		This work was supported by JSPS KAKENHI Grant Number JP25KJ1818.
	\end{acknowledgements}

\end{document}